\documentclass[10pt]{article}

\usepackage{amsmath, amssymb, amsthm} 
\usepackage{color}
         \newtheorem{theorem}{Theorem}
	 \newtheorem{proposition}[theorem]{Proposition}
	 \newtheorem{lemma}[theorem]{Lemma}
         \newtheorem{reform}{Problem}

\theoremstyle{definition}        
         \newtheorem{remark}[theorem]{Remark}

\newcommand \Rbb {\mathbb{R}}
\newcommand \Tbb {\mathbb{T}}
\newcommand \Cbb {\mathbb{C}}
\newcommand \Dbb {\mathbb{D}}

\newcommand \be {\begin{equation}}
\newcommand \ee {\end{equation}}

\newcommand \si {\sigma}
\newcommand \ka {\kappa}
\newcommand \om {\omega}

\def \t {\tilde} 

\newcommand \pa {\partial}

\newcommand \al {\alpha}
\newcommand \de {\delta}
\newcommand \ep {\epsilon}

\newcommand \Ga {\Gamma}

\newcommand \Acal {\mathcal{A}}
\newcommand \Ccal {\mathcal{C}}

\newcommand \Scal {\mathcal{S}}

\DeclareMathOperator\Real{Re}
\DeclareMathOperator\Imag{Im}

\title{On the existence of stationary splash singularities for the Euler
  equations}
\author{Diego C\'ordoba\thanks{dcg@icmat.es}, Alberto Enciso\thanks{aenciso@icmat.es} and Nastasia Grubic\thanks{nastasia.grubic@icmat.es}}
\date{\normalsize Instituto de Ciencias Matem\'aticas\\ Consejo Superior de
  Investigaciones Cient\'\i ficas\\ 28049 Madrid, Spain}

\begin{document}


\maketitle
 
\begin{abstract} 
  In this paper we discuss the existence of stationary incompressible fluids with splash singularities. Specifically, we
  show that there are stationary solutions to the Euler equations with
  two fluids whose interfaces are arbitrarily close to a
  splash, and that there are stationary water waves with splash
  singularities.
\end{abstract}

 


\section{Introduction}

   The aim of this paper is to seek for stationary splash singularities for a two-fluid interface. We consider the two-fluid
	incompressible irrotational Euler equations in $\Rbb^2$ where the interface 
	$$
\Scal = \{z(\al) = (z_1(\al), z_2(\al)) \,|\, \alpha\in \Rbb\}
$$ 
	separates the plane in two regions $\Omega_j$, with
        $j=1,2$. Each $\Omega_j$ denotes the region occupied by the
        two different fluids with velocities $v^j=(v_1^j,v_2^j)$,
        different constant densities $\rho_j$ and pressures $p^j$,
        which must satisfy the following equations:
	
\begin{subequations}\label{equations}
\begin{align}
\rho_j (v^j \cdot \nabla )v^j = - \nabla p^j - g \rho_j \, e_2 \quad &\text{in} \quad \Omega_j,\\
\nabla\cdot v^j=0  \quad \mathrm{and} \quad \nabla^{\perp}v^j=0  \quad &\text{in} \quad \Omega_j, \\
v^j \cdot (\pa_\al z)^\perp = 0 \quad &\mathrm{on} \quad \Scal,\\
p^1 - p^2 = -\sigma K \quad &\mathrm{on} \quad \Scal.
\end{align} 
\end{subequations}
Here $j\in\{1,2\}$, $\sigma>0$ is the surface tension coefficient,
$e_2$ is the second vector of a Cartesian basis and
$K$ is the curvature of the interface. 
	
	For the time dependent case of water waves (that is, $\rho_1 =0$, the
        fluid is irrotational and no surface tension),
        Castro--Cordoba--Fefferman--Gancedo--Gomez-Serrano \cite{CCFGG1} showed the formation  in finite time of splash and splat singularities. A splash singularity is when the interface remains smooth but self-intersects at a point and a splat singularity is when it self-intersects along an arc. The system of water waves in $\Rbb^2$ can be written in terms of the boundary $\Scal(t) = \{z(\al,t) = (z_1(\al,t), z_2(\al,t)) \,|\, \alpha\in \Rbb\}$ and the measure of the vorticity in the boundary $\nabla^\perp \cdot v  =  \om(\al,t)\delta (x - z(\al,t))$. Notice that for these types of singularities to happen the amplitude of the  vorticity $\om(\al,t)$ has to blow-up at the point of the self-intersection.   With a different approach Coutand--Shkoller proved  in \cite{CS1} that these two types of singularities develop when you consider non zero vorticity. The presence of surface tension, see \cite{CCFGG2}, does not prevent the formation of a splash and a splat singularity. 
	
	In a recent work by Fefferman--Ionescu--Lie \cite{FIL} they
        showed that the presence of a second fluid ($\rho_2,\rho_1>0$)
        prevents both splash and splat singularities to form in finite
        time. The condition $\rho_2,\rho_1>0$ is used to show a
        critical $L^{\infty}$ bound for the measure  of the vorticity
        in the boundary. With a different approach, similar results
        have been obtained by Coutand--Shkoller \cite{CS2} for the vortex sheet.

        The main goal of this paper is not to study the dynamics in
        which a non-intersecting smooth curve self-intersects in
        finite-time but to show the existence of stationary solutions
        which have a splash singularity. It is worth recalling that
        in the last few years the properties of stationary solutions
        to the Euler equation have attracted considerable attention,
        which has led to recent results concerning e.g.\ Liouville
        theorems~\cite{Nadirashvili,Chae}, connections with coadjoint
        orbits in the space of volume-preserving
        diffeomorphisms~\cite{Sverak}, the existence of weak
        solutions~\cite{CS}, mixing properties for the flow~\cite{KKP} and the existence of knotted vortex lines
        and thin vortex tubes~\cite{EP12,EP15}.

  Specifically, in this paper we will prove two related results
  concerning the existence of splash singularities in stationary
  solutions of the Euler equations. In both cases we will assume that
  the fluid is {\em periodic}\/, i.e., that the pressure, the velocity of the fluid
  and the geometry of the interface remain invariant under the
  translation
\begin{equation*}
(x,y)\mapsto (x+2\pi,y)\,.
\end{equation*}
Suitable boundary conditions are also specified as $y\to\pm\infty$;
full details are given in Section~\ref{sec:1}.

The first result asserts that there are almost-splash stationary
solutions to the Euler equations with two fluids. To make precise what
we understand by ``almost splash'', consider a curve $z(\al)$
parametrized by an arc-length parameter~$\al$. For positive but
arbitrarily small~$\eta$, we say that it is an {\em
  $\eta$-splash curve}\/ if
\begin{equation*}
\inf_{\al<\beta}\frac{|z(\al)-z(\beta)|}{\min\{\al-\beta,1\}}\leq \eta
\end{equation*}
and if moreover this condition is only satisfied when $\al$
and~$\beta$ respectively lie in intervals $[\al_1,\al_2]$ and $[\beta_1,\beta_2]$ whose lengths
are bounded by some quantity that tends to zero as $\eta\searrow0$
(modulo periodicity). The precise (but fixed) bound that we take is not important.

Intuitively speaking, this means that there are two points in the
curve whose chord distance (that is, their distance as measured in the
ambient space~$\Rbb^2$) is at most~$\eta$ while their distance as
measured along the curve is not small. The condition that the curve be
parametrized by arc-length ensures that this actually corresponds to
the intuition that the curve is close to self-intersecting, while the
assertion about the condition being satisfied only for $\al$ and $\beta$
in two small intervals guarantees that the shape of the $\eta$-splash curves
actually tend to a ``splash'' as $\eta\searrow0$, and not to other
sort of self-intersecting curve such as a splat curve. In particular,
a splash curve can be defined as an $\eta$-splash curve with $\eta=0$.

With this definition, our first result can then be stated as follows:

\begin{theorem}\label{T.two}
  Let us fix the density of the second fluid $\rho_2>0$ and consider
  any $\eta>0$. For any small enough $\rho_1\geq0$ and $g$ there is
  some positive surface tension coefficient $\si$ such that there is a
  periodic solution to the two-fluid problem~\eqref{equations} for
  which the interface $\Scal$ is an $\eta$-splash curve.
\end{theorem}

It should be noticed that Theorem~\ref{T.two} does not contradict the
result of Fefferman--Ionescu--Lie~\cite{FIL} on the absence of splash
singularities when $\rho_1$ is positive. Of course, the result for the
2D Euler equations has also a bearing on the
three-dimensional case with periodic boundary conditions, as it
corresponds to 3D solutions that are invariant under translations
along the third axis.

Our second theorem concerns the existence of stationary splash
singularities for water waves, that is, for $\rho_1=0$. Essentially,
what we prove is that in this case one can take $\eta=0$. These
solutions will necessarily have an unbounded amplitude of the
vorticity.

\begin{theorem}\label{T.one}
  Let us fix the density of the second fluid $\rho_2>0$ and assume that $\rho_1=0$. Then for any small enough $g$ there is
  some positive surface tension coefficient $\si$ such that there is a
  periodic solution to the water waves problem~\eqref{equations} for
  which the interface $\Scal$ has a splash singularity.
\end{theorem}

The idea of the proof of this result is to perturb a family of exact
stationary water waves introduced by
Crapper~\cite{crapper}. The idea of perturbing Crapper waves to obtain
more general stationary water waves was introduced by Akers, Ambrose
and Wright in~\cite{AAW} and exploited further
in~\cite{boeck}. Global bifurcation results for traveling waves \`a la
Rabinowitz, also
considering the presence of two fluids, were presented in \cite{ASW}.

In this paper we are interested in the existence of splash or
almost-splash solutions to the Euler equations with two
fluids (of course, this involves ensuring that
these curves remain weak solutions to the equations~\eqref{equations}). To this end we will make use of the fact
that the family of Crapper waves, which depends on a parameter $A$,
includes a splash curve for certain value of the parameter, $A_0$, so
we will try to control the effects of the perturbation in a
neighborhood of this solution. In a perturbation argument, this
involves controlling the perturbations not only of the solutions with
parameter $A_0$, but also with parameters slightly smaller or larger. Notice the connection
between the Crapper solutions with $A>A_0$ and the Euler equations is
uncertain, so in the case of water waves we will eventually resort to the trick
of ``opening up'' the domain using a suitable conformal
transformation. The proofs of Theorems~\ref{T.two} and~\ref{T.one} are
respectively given in Sections~\ref{S.two} and~\ref{S.one}, while in
Section~\ref{sec:1} we reformulate the problem in a way that is most
convenient for our purposes.

\section{Setting up the problem}
\label{sec:1}

\subsection{The vorticity formulation}

Throughout the paper the interface 
$$
\Scal = \{z(\al) = (z_1(\al), z_2(\al)) \,|\, \alpha\in \Rbb\}
$$ 
will be parametrized by a function $z(\al)$ that satisfies the
periodicity conditions  
\begin{equation*}
z_1(\alpha + 2\pi) = z_1(\alpha) + 2\pi, \quad  z_2(\alpha + 2\pi) = z_2(\alpha)
\end{equation*}
and is symmetric with respect to the $y$-axis:
\begin{equation*}
z_1(-\alpha) = -z_1(\alpha), \quad  z_2(-\alpha) = z_2(\alpha).
\end{equation*}
The interface $\Scal$ separates $\Rbb^2$, which we will often identify
with the complex plane $\Cbb$, into two disjoint unbounded domains
$\Omega_j$ ($j=1,2$), each of which is occupied by a fluid of constant
(but distinct) density $\rho_j$. In particular, the interface is given
by the set of discontinuity of the density function      
$$
\rho = \begin{cases}
        \rho_1, \quad z\in \Omega_1,
        \\
        \rho_2, \quad z\in \overline{\Omega}_2 = \Rbb^2\setminus\Omega_1.
       \end{cases}
$$
We assume that the densities satisfy $\rho_2>0$ and $\rho_1\geq 0$. The
fluid flow, governed by the stationary Euler equations  
\be
\label{euler}
\rho (v \cdot \nabla )v = - \nabla p - g \rho \, e_2,
\ee
where $p$ is the pressure and $g$ the gravity constant, is assumed incompressible
\be
\label{incomp}
\nabla \cdot v = 0.
\ee
Both equations are to be understood weakly, since the density $\rho$ is discontinuous. However, on either side of the interface, the velocity is smooth, so restrictions $v^j = v|_{\Omega_j}$ satisfy the above equations strongly in their respective domains. Away from the interface, we assume that
\be
\label{v2inf}
v^2 \rightarrow -1 \quad \mathrm{uniformly \ as} \quad y \rightarrow -\infty
\ee
and
\be
\label{v1inf}
v^1 \rightarrow 0 \quad \mathrm{uniformly \ as} \quad y \rightarrow \infty.
\ee
Here and in what follows, we identify $\Rbb^2$ with the complex plane
whenever it is notationally convenient and use the notation $z=x+i
y$. These conditions guarantee that the mean value of the vorticity
amplitude is also well behaved in the limit $\rho_1 \rightarrow 0$. We
also impose the kinematic boundary condition  
\be
\label{kbc}
v^j \cdot (\pa_\al z)^\perp = 0 \quad \mathrm{on} \quad \Scal,
\ee
and we assume the velocity is irrotational in the interior of each
domain, i.e.,
\be
\label{irrot}
\nabla^\perp \cdot v^j = 0 \quad \mathrm{in} \quad \Omega_j.
\ee
These are the equations~\eqref{equations} mentioned in the
Introduction and the pertinent boundary conditions.

In particular, the vorticity 
$$
\om = \nabla^\perp \cdot v
$$ 
can be seen as a measure supported on $\Scal$. With some abuse of
notation, we will write this
measure as
$$
\om (z) = \om(\al)\delta (z - z(\al)),
$$
so $\om(\al)$ will stand for the amplitude of the vorticity along the interface.
By the incompressibility condition, the velocity can be written as a
normal gradient of a stream function $v = \nabla^\perp \psi$, which
therefore satisfies the equation
$$
\triangle \psi = \om.
$$
Hence away from the interface we have
$$
\psi(z) = \frac{1}{2\pi} \int \ln |z - z(\al)| \, \om(\al)d\al.
$$
This function is harmonic in the interior of each domain and
continuous over the interface. Condition \eqref{kbc} implies it is
actually constant on $\Scal$ so without loss of generality we can
assume that it is zero: $\psi|_{\Scal}=0$. To obtain the velocities $v^j$, we take the normal gradient of $\psi$ in the respective domain. Approaching the interface from inside of each $\Omega_j$ in the normal direction, we have
\begin{align*}
&v^2(z(\al)) = BR(z,\om)(\al) + \frac{1}{2}\frac{\om(\al)}{|\pa_\al z(\al)|^2}\pa_\al z(\al),
\\
&v^1(z(\al)) = BR(z,\om)(\al) - \frac{1}{2}\frac{\om(\al)}{|\pa_\al z(\al)|^2}\pa_\al z(\al),
\end{align*}
where $BR(z,\om)$ denotes the Birkhoff--Rott integral of $\om$ along
the curve $z(\al)$, that is,
$$
BR(z,\om)(\al) = \frac{1}{2\pi} \mathrm{PV} \int \frac{(z(\al) - z(\beta))^\bot}{|z(\al) - z(\beta)|^2}  \, \om(\beta)d\beta.
$$
In particular, the amplitude of the vorticity measures the jump in the tangential component of the velocity along the interface, 
$$
(v^2-v^1)\cdot \pa_\al z  = \om.
$$
It remains to rewrite the Euler equation \eqref{euler} and kinematic
boundary condition \eqref{kbc} in terms of the vorticity amplitude
$\om(\al)$ and the parametrization $z$. The latter now reads 
\be
\label{BR:nor}
BR(z,\om) \cdot \pa^\bot_\al z = 0 \quad \mathrm{on} \quad \Scal,
\ee
while the Euler equation can be integrated to obtain the Bernoulli equation 
$$
\rho_j\Bigg(\frac{1}{2}|v^j|^2 + gy\Bigg) + p_j = \rho_j k_j \quad \mathrm{in} \quad \Omega_j, 
$$
for $k_j$ constant. Let us take $k_1 = k_2 = \frac{1}{2} +
\kappa$. This way, we can understand $\kappa$ as a perturbation of the
constant arising in the pure capillary wave problem (cf.\
\cite{OS}). We assume that on the interface, the pressures are related by     
$$
p_2 = p_1 + \si K,
$$
where $\si> 0$ is the surface tension coefficient and $K(z)$ is the
curvature of the free boundary, i.e., 
$$
K(z) = \frac{\pa_\al z_1 \pa_\al^2 z_2 - \pa_\al z_2 \pa_\al^2 z_1 }{\big[(\pa_\al z_1 )^2 + (\pa_\al z_2)^2\big]^{3/2}}.
$$
Subtracting the two equations, we then obtain 
$$
\frac{1}{|\pa_\al z(\al)|^2} \Bigg(\big(v^2\cdot\pa_\al z\big)^2 + \frac{\rho_1}{\rho_2-\rho_1}\big((v^2 - v^1)\cdot \pa_\al z \big)\big((v^2 + v^1)\cdot \pa_\al z \big) \Bigg) = F(z)
$$
where $F(z)$ is given by 
$$
F(z) = -\frac{2\rho_2}{\rho_2- \rho_1} q K(z) - 2 g z_2 + 1 + 2 \ka
$$
and $q:= \frac{\si}{\rho_2}$. Setting 
$$
\epsilon := \frac{2\rho_1}{\rho_2-\rho_1},
$$ 
and then using the expressions for the velocity on the interface, we have
$$
\big(\om + 2BR(z,\om)\cdot \pa_\al z\big)^2 + 4\epsilon\om \, BR(z,\om)\cdot \pa_\al z = 4|\pa_\al z|^2 F(z).
$$
Finally, collecting all the non-local terms on the right-hand side, we have  
\be
\label{lhs}
\Big((1 + \epsilon)\om + 2BR(z,\om)\cdot \pa_\al z\Big)^2  = 4\,|\pa_\al z|^2 F(z) + \epsilon\,(2 + \epsilon )\,\om^2,
\ee
where, in terms of $\epsilon$ the function $F(z)$ can be written as   
\be
\label{rhs}
F(z) := -(2 + \epsilon) q\, K(z) - 2 g z_2 + 1 + 2\ka.
\ee

\subsection{Choice of parametrization for $z$}
\label{subsec:par}

We still have one degree of freedom left. The equations are only
determined up to a parametrization of the interface. To fix the
parametrization, we use the hodograph transform with respect to the
lower fluid. Since we will only consider the lower fluid, in this
section we omit the subscript two in all the quantities. The idea is
to use the analytic function $w = \phi + i \psi$ as an independent
variable instead of $z = x + iy$. This will allow us to transform the
free boundary problem into an equation on a fixed domain. In fact, as long as $\Scal$ is non-self-intersecting, it can be shown that 
$$
w = \phi + i \psi : \Omega \rightarrow \Cbb_-
$$
is a conformal bijection, which extends to a homeomorphism on $\overline{\Omega}\rightarrow \overline{\Cbb}_-$ (cf.\ \cite{ST}). Here, $\Cbb_-$ denotes the lower half-plane and we have used that
$$
\psi(z(\al)) = 0.
$$
Recall that $\psi$ is constant on the interface by the kinematic
boundary condition \eqref{kbc}. Moreover, the periodicity assumption on the interface and condition \eqref{v2inf} imply 
$$
\aligned
&\psi(x+2\pi, y) = \psi (x,y),
\\
&\phi(x +2\pi, y) = \phi (x, y) + 2\pi,
\endaligned
$$  
and we can actually restrict attention to a period (cf.\ \cite{ST}).
In particular, the inverse mapping 
$$
\phi \mapsto w^{-1}(\phi, 0)
$$ 
gives a useful (although not arc-length) parametrization of the interface $\Scal$. We will
henceforth assume that the curve $z(\al)$ is parametrized in this
manner, which is equivalent to requiring   
%
\be
\label{parametrization}
\phi(z(\al)) = \al.
\ee
Hence we obtain a simple relation for the velocity and tangent vector on the interface:
$$
v^2 \cdot \pa_\al z = \nabla \phi \cdot \pa_\al z = 1.
$$
This translates into a particularly nice expression in terms of the
amplitude of vorticity:
$$
\om + 2BR(z,\om)\cdot \pa_\al z = 2.
$$
Using the above relations and rearranging, equation \eqref{lhs} reads   
$$
2 - 2|\pa_\al z|^2 F(z) = \epsilon\,\om (\om - 2).
$$
We now have a system of three equations in terms of the vorticity and
the parametrization, and the problem can be reformulated as: 
\begin{reform}
Find $2\pi$-periodic functions $\om(\al)$ and $z(\al) - \al$ satisfying 
\begin{subequations}
\label{eq:BR}
\begin{align}
2|\pa_\al z|^2 F(z)+ \epsilon\,\om (\om - 2)&=2, \\
 2BR(z,\om)\cdot \pa_\al z + \om &= 2, \\
 BR(z,\om) \cdot \pa^\bot_\al z &= 0,
\end{align}
\end{subequations}
where $F$ is given by \eqref{rhs}.
\end{reform}
This problem can be simplified further, since the last two equations
of \eqref{eq:BR} have been combined to construct the
parametrization. In fact, we can write the velocity in terms of the
polar coordinates associated with~$w$ as
$$
(\nabla \phi) \circ w^{-1} = e^{ - i f }.
$$
This defines a function
$$
f= \theta + i \tau : \Cbb_- \rightarrow \Cbb
$$
that is analytic in $\Cbb_-$ and continuous in
$\overline{\Cbb}_-$. Note that the boundary conditions ensure
$$
f \rightarrow 0 \quad \mathrm{as} \quad \psi\rightarrow -\infty.
$$ 
Equation \eqref{parametrization} now implies
$$
\pa_\al z = |\pa_\al z|e^{i\theta} = e^{if}.
$$
This has two advantages. First, the expression for curvature takes the simple form   
$$
K = e^\tau \frac{d\theta}{d\alpha}.
$$
Second, as $\theta$ and $\tau$ are a $2\pi$-periodic conjugate
functions, they must be related on $\Scal$ by the periodic Hilbert transform, that is, $\tau(\al)$ can be written as a function of $\theta(\al)$ via   
$$
\tau(\al) = H\theta(\al) = \frac{1}{2\pi}\mathrm{PV}\int_{-\pi}^\pi \cot\frac{\al - \al '}{2}\,\theta(\al') d\al '.
$$
We therefore take $\theta(\al)$ as our main unknown and consider $z$
as a function of $\theta$ via the integral operator 
\be
\label{zoftheta}
z(\alpha) = I(\theta)(\al) := \int_{-\pi}^\al e^{-H\theta(\al ') + i \theta (\al')} d\al'.
\ee
It remains to rewrite $F(z)$ in terms of $\theta$. We have  
$$
F(z) = -2q\Big(1 + \frac{\epsilon}{2}\Big)e^{\tau}\frac{d\theta}{d\al}
- 2 g \, \Imag I(\theta) + 1 + 2\ka,  
$$
so after a short calculation, the first equation in \eqref{eq:BR} reads  
$$
q \Big(1 + \frac{\epsilon}{2}\Big)\frac{d\theta}{d\al} + \sinh \tau =
- g e^{-\tau} \,\Imag I(\theta) + \ka e^{-\tau}  + \frac{e^\tau}{4}\epsilon\,\om (\om - 2). 
$$
Using $\tau = H\theta$, the problem reduces to 

\begin{reform}\label{ProbA}
Find $2\pi$-periodic functions $\theta(\al)$ and $\om(\al)$, such that
$\theta$ is odd, $\om$ is even and they satisfy 
\begin{subequations}
\label{eq:BR1}
\begin{align*}
 q \Big(1 + \frac{\epsilon}{2}\Big)\frac{d\theta}{d\al} + \sinh
   H\theta + g e^{-H\theta}\, \Imag I(\theta) - \ka e^{-H\theta} - \frac{\epsilon}{4}e^{H\theta}\om (\om - 2) &= 0, \\
2BR(z,\om)\cdot \pa_\al z + \om &= 2,
\end{align*}
\end{subequations}
where $z := I(\theta)$ is defined by \eqref{zoftheta}. 
\end{reform}

Note that the integration constant in \eqref{zoftheta} does not affect
the Birkhoff--Rott integral. However, it serves to fix the origin and contributes a constant to the Bernoulli equation.

%

\subsection{Crapper solutions}

The equations that appear in Problem~\ref{ProbA} depend on four
parameters $g, \epsilon, \ka$ and $q$. Recall that
$q$ represents the surface tension, $g$ is the gravity and $\epsilon$
indicates the presence of the upper fluid. Setting $\epsilon$ to zero,
the equations decouple and we recover the capillary-gravity wave
problem as studied in \cite{AAW}. Once a solution $\theta$ of the
Bernoulli equation is known, the vorticity can be recovered from the
second equation as long as the underlying interface is non
self-intersecting (cf.\ Proposition \ref{prop:A} below). If, in
addition, we set $g = 0$, we recover the pure capillary waves problem
as formulated by Levi-Civita (see e.g.\ \cite{OS}), namely,

\begin{reform}\label{ProbC}
Find a $2\pi$-periodic, analytic function  $f = \theta + i \tau$ on the lower half-plane that satisfies 
\be
\label{ref:2}
q \frac{d\theta}{d\al} =  -\sinh H\theta
\ee
on the boundary and tends to zero at infinity. 
\end{reform}

This problem admits a family of exact solutions depending on the parameter $q$. In fact, Crapper has shown \cite{crapper} that the family of analytic functions    
\be
\label{ref:3}
f_A(w) := 2i \log \frac{1 + Ae^{-iw}}{1 - Ae^{-iw}} 
\ee
has all the required properties. Here the parameter $A$, which is a
real number smaller than 1 in absolute value, depends on $q$ via 
$$
\quad q = \frac{1 + A^2}{1-A^2}.
$$
To obtain the wave profiles, recall that by construction we have
$$
\pa_\al z = e^{ - \tau_A  + i \theta_A} = \Big(\frac{1 - Ae^{-i\al}}{1 + Ae^{-i\al}}\Big)^2
$$
on the boundary, so integrating we obtain the parametrization of the interface 
$$
z_A(\alpha) = \alpha + \frac{4i}{1 + Ae^{-i\al}} - 4i.
$$
The constant has been chosen to have $z_A (\al)=\al$ for $A=0$. It
actually suffices to consider $A\geq 0$, since the transformation $A
\mapsto -A$  corresponds to a translation $\al \rightarrow \al +
\pi$. For sufficiently large values of parameter $A$ these solutions
can no longer be represented as a graph of a function, and eventually
self-intersect. It is not hard to see that the curve $z_A$ can be
represented as the graph of a function (that is, as $y = h(x)$) if and
only if
 $$
 A<\sqrt{2} -1,
 $$
and that $z_A(\al)$ does not have self-intersections if and only if
 $$
 A< A_0 \approx 0.45467.
 $$
For $A=A_0$, the curve $z_A(\al)$ exhibits a splash, while for $A$
slightly larger than $A_0$ the curve intersects at exactly two points,
and the intersection is transverse. (What is not clear, of course, is
the connection between curves with $A>A_0$ and the Euler equations.)

\section{Proof of Theorem~\ref{T.two}}
\label{S.two}

In this section, we construct stationary solutions to the Euler
equations with two fluids whose interface is arbitrarily close to a
splash. For this, we use the implicit function theorem to perturb with
respect to the density of the upper fluid (and the gravity) a
Crapper solution that is arbitrarily close to the splash. This
strategy has been employed to perturb Crapper solutions with respect
to gravity in the water wave setting \cite{AAW} and then to deal with
presence of a constant vorticity in a fluid of finite depth \cite{boeck}. 
%
%

Let $H_{\rm even}^s$ and $H_{\rm odd}^s$ respectively denote the subspace
of even and odd functions in the Sobolev space $H^s(\Tbb)$. We define  
\begin{equation}\label{defG}
G= (G_1, G_2): H_{\rm odd}^2 \times H_{\rm even}^1 \times \Rbb^3 \rightarrow H^1_{\rm even} \times H_{\rm even}^1
\end{equation}
to be the differentiable mapping of components
\begin{align*}
G_1(\theta, \om; \,\epsilon, g, \kappa) &:= q \Big(1 +
                                          \frac{\epsilon}{2}\Big)\frac{d\theta}{d\al}
                                          + \sinh H\theta + g
                                          e^{-H\theta} \,\Imag I(\theta) - \frac{\epsilon}{4}e^{H\theta}\om (\om - 2) - \ka e^{- H\theta}
\\
G_2(\theta, \om; \,\epsilon, g, \kappa) &:=  \om + 2BR(z,\om)\cdot \pa_\al z - 2,
\end{align*}
where $z$ is considered to be a function of $\theta$ via $z=I(\theta)$
(cf.~\eqref{zoftheta}).

In order to use the implicit function theorem on Crapper solutions
$\theta_A$ we first need to construct the corresponding vorticity
$\om_A$. The fact that $\om_A$ is well defined is a consequence of the
following proposition, where $1$ stands for the identity operator. For
the proof of this result see e.g.\ \cite{BMO} or \cite{CCG}. 

\begin{proposition}
\label{prop:A}
Let $z\in H^3$ and assume $z$ is a curve without
self-intersections. Then 
$$
\Acal(z)(\omega) = 2BR(z,\om)\cdot \pa_\al z
$$
defines a compact linear operator 
$$
\Acal(z) : H^1\rightarrow H^1
$$ 
whose eigenvalues are strictly smaller than $1$ in absolute value. In particular, the operator $1 + \Acal(z)$ is invertible.  
\end{proposition}
We will sometimes denote points in the domain of definition of $G$ by
$$
\xi = (\theta, \om;\, \epsilon, g, \kappa).
$$ 
In particular, for the Crapper solutions we set   
$$
\xi_A = (\theta_A, \om_A;\, 0, 0, 0)
$$ 
so by construction we have 
$$
G(\xi_A) = 0.
$$
It is clear that the map~$G$, defined in~\eqref{defG}, is
differentiable and that its Fr\'{e}chet derivative with respect to $(\theta, \om)$ has the simple triangular form  
\[ 
  D_{\theta, \om}G(\xi_A) = \left(
			  \begin{array}{cc}
                           \Ga & 0 \\
			   D_\theta G_2	& D_\om G_2   
                          \end{array}
		    \right)\,,
\]
where $\Ga$ is just the Fr\'{e}chet derivative of the pure capillary wave operator \eqref{ref:2},
$$
\Ga u = q \frac{du}{d\si} + \cosh(H\theta_A)\, Hu,
$$
and $(D_\om G_2)T$ is exactly 
$$
(D_\om G_2)T = (1 + \Acal(z_A))T 
$$
since $G_2$ is linear in $\om$. By Proposition \ref{prop:A}, we
already know that $D_\om G_2$ is invertible on $H^1_{\rm even}$ as
long as $z_A$ does not self-intersect. The structure of $\Ga$ is
already known as well. In fact, we have the
following result, which we borrow from \cite{AAW}  and whose proof we sketch here for completeness:

\begin{lemma}
\label{lem:Ga}
The operator $\Ga:H^1_{\rm odd} \rightarrow L^2_{\rm even}$ is
injective, but not surjective. Its cokernel is spanned by the function
$\cos\theta_A$, which is orthogonal to the image of~$\Ga$.
\end{lemma}

\begin{proof}
First note that $\Ga$ is a compact perturbation of Fredholm operator
of index $-1$, and as such it is another Fredholm operator of index
$-1$ itself. In order to see this, notice that $\Ga$ can be written as a sum 
\begin{equation*} 
\Ga=\Ga_1+K_1,\qquad
  \Ga_1 u:= q \frac{du}{d\si},
\qquad
  K_1 u:=  \cosh(H\theta_A)\, Hu,
\end{equation*}
where $\Ga_1$ is Fredholm of index $-1$ and $K_1$ is compact. Indeed,
$\Ga_1$ is injective on $H^1_{\rm odd}$, since constants are the only
trivial solutions. On the other hand, the dimension of the cokernel is
one, since $\Ga_1T = f$ has a periodic odd solution if and only if $f$
is even and of zero mean ($\int_{-\pi}^{\pi}f(\si)d\si = 0$). 
The operator $K_1$ is compact because the embedding $H^1_{\rm odd} \hookrightarrow L^2$ is compact by the Rellich--Kondrachov theorem.

$\Ga$ is injective when restricted to $H^1_{\rm odd}$, since zero is an eigenvalue of $\Ga:H^1\rightarrow L^2$ with geometric multiplicity one (cf.\ \cite{OS}). The only eigenfunction is given by $d\theta_A/d\si$, which is even.

Hence the cokernel is at most one-dimensional. To show that it is
spanned by $\cos\theta_A$, note that for any function $\theta$ we
have, with $\ep=g=0$,
\be\label{ident}
\int_{-\pi}^{\pi} G_1(\xi) \cos \theta \,d\al = \int_{-\pi}^{\pi} \Big[q \frac{d\theta}{d\al} + \sinh H\theta\Big] \cos \theta \,d\al = 0 
\ee
since the first term is just the integral of the derivative of
$q\,\sin \theta$, so it integrates to zero. To show that we also have
$\int_{-\pi}^\pi \sinh H\theta\, \cos \theta\, d\al=0$, it suffices to
employ the identity  
$$
\int_{-\pi}^{\pi} e^{\pm H\theta(\al) + i\theta(\al)}d\al = 2\pi,
$$
which follows from the Cauchy integral theorem. 

Hence taking the derivative with respect to $\theta$ at $\theta_A$ in
the identity~\eqref{ident} we obtain
$$
0 = \int_{-\pi}^{\pi} \Ga u\, \cos\theta_A \,d\al + \int_{-\pi}^{\pi}\Big[q \frac{d\theta_A}{d\al} + \sinh H\theta_A\Big]  \sin\theta_A \,u \,d\al.
$$
The expression in square brackets zero by construction, so we obtain
$$
\langle\Ga u, \cos\theta_A\rangle = 0
$$
for all~$u$, where the angle brackets denote the $L^2$ product.
\end{proof}

In view of the above, we obtain the following result:

\begin{proposition}
\label{FrDer}
Let $A < A_0$. Then the Fr\'{e}chet derivative of $G$ with respect to $(\theta, \om)$ at $\xi_A$ is injective, i.e.,
$$
\ker D_{\theta, \om}G =\{0\}.
$$
On the other hand, the equation 
\[ 
                    \left(
			  \begin{array}{cc}
                           \Ga & 0 \\
			   D_\theta G_2	& D_\om G_2   
                          \end{array}
		    \right)
		    \begin{pmatrix}
                            u \\
			    T   
                    \end{pmatrix} 
                    = 
			  \begin{pmatrix}
                           f \\
			    g   
                          \end{pmatrix}
\]
has a solution if and only if 
$$
\langle f, \cos\theta_A \rangle= 0.
$$ 
\end{proposition}

\begin{proof}
 $D_{\theta, \om}G$ is injective since both $\Ga$ and $D_{\om}G_2$
 are. On the other hand, if 
 $\langle f, \cos \theta_A\rangle = 0$, then there exists $u$ such that 
 $$
 \Ga u = f,
 $$
 by Lemma \ref{lem:Ga}. Since $D_{\om}G_2$ is invertible by Lemma \ref{prop:A}, there exists $T$ such that 
 $$
 (D_{\om}G_2) T = g - (D_{\theta}G_2) u,
 $$
 so the claim follows.
\end{proof}

Since $D_{\theta, \om}G$ is not surjective, we cannot use the implicit
function theorem directly to prove Theorem~\ref{T.two} (the remaining hypothesis of the 
implicit function theorem are easy to check). Therefore,
following \cite{AAW}, we use an adaptation of the Lyapunov--Schmidt
reduction argument. So let $\Pi$ denote the $L^2$ projector onto the
linear span of $\cos \theta_A$, i.e.,
$$º
\Pi u := \langle \cos\theta_A, u \rangle \frac{\cos\theta_A}{\|\cos\theta_A\|^2_2}. 
$$
Then $Q : = 1 - \Pi$ is the projector on $\Ga(H^2_{\rm odd})$ and $\t G = (Q G_1, G_2)$ is a mapping
$$
\t G : H^2_{\rm odd} \times H^1_{\rm even} \times \Rbb^3 \rightarrow \Ga(H^2_{\rm odd}) \times L^2,
$$
whose Fr\'{e}chet derivative with respect to $(\theta, \om)$ at
$\xi_A$ is an isomorphism by construction. We can then apply the
implicit function theorem on $\t G$ to obtain a smooth function $\Theta_A(\epsilon, g, \kappa)$ satisfying $\Theta_A(0, 0, 0) = (\theta_A, \om_A)$ and
$$
\t G(\Theta_A(\epsilon, g, \kappa); \, \epsilon, g, \kappa) = 0
$$
for all $(\ep,g,\kappa)$ in a small neighbourhood of $(0, 0, 0)$. This
function does not necessarily satisfy the original equation $G = 0$,
since the projection $\Pi G_1$ on $\cos\theta_A$ is not necessarily
zero. So consider  the differentiable real-valued function defined in a neighborhood
of the origin $(0,0,0)\in\Rbb^3$ as
$$
f(\epsilon, g;\, \kappa) = \langle \cos\theta_A,\, G_1(\Theta_A(\epsilon, g, \kappa); \epsilon, g, \kappa )\rangle
$$
and note that $f(0,0;0) = 0$. For the derivative at $(0,0,0)$, we have 
$$
\pa_\kappa f(0,0; 0) = \langle\cos\theta_A,\, \Ga \pa_\kappa \Theta_A + \pa_\kappa G_1  \rangle = -2\pi,
$$
where we have used that the first term is zero by Lemma~\ref{lem:Ga} and
that the second gives
$$
\langle \cos \theta_A, \, \pa_\kappa G_1\rangle = \int_{-\pi}^\pi e^{-H\theta_A(\si)}\cos\theta_A(\si) d\si = -2\pi
$$
by the Cauchy integral theorem. We can apply the implicit function
theorem once again to obtain a smooth function $\kappa_*(\epsilon, g)$
satisfying $f(\epsilon, g; \, \kappa_*(\epsilon, g)) = 0$ and
$\kappa_*(0, 0) = 0$. Therefore, denoting by $B_r$ the interval of reals of
absolute value smaller than~$r$, we have proved the following

\begin{theorem}
\label{main:thm}
Let $A < A_0$. Then there exist $\epsilon_A>0$, $g_A>0$, $\kappa_A>0$,
a unique smooth function 
$$
\kappa_* : B_{\epsilon_A} \times B_{g_A}  \rightarrow B_{\kappa_A}  
$$
such that $\kappa_*(0,0) = 0$ and a unique smooth function 
$$
\Theta_A : B_{\epsilon_A} \times B_{g_A} \times B_{\kappa_A}  \rightarrow H^2_{\rm odd}\times H^1_{\rm even} 
$$
such that $\Theta_A(0,0,0) = (\theta_A, \om_A)$ and    
$$
G\big(\Theta_A(\epsilon, g, \ka_*(\epsilon, g));\, \epsilon, g, \ka_*(\epsilon, g)\big) = 0.
$$
\end{theorem}

\begin{remark}
\label{rem:1}
The restriction $A<A_0$ is necessary if the upper fluid is present,
since absolute values of eigenvalues of $\Acal(z)$ approach $1$ as the
arc-chord condition fails (cf.\ \cite{CCG}). If the upper fluid is
neglected, the equations \eqref{eq:BR1} decouple and we recover the
capillary-gravity waves problem as studied in \cite{AAW}. In
particular, if we were only concerned with the function $G_1$ (as
opposed to $G=(G_1,G_2)$), the proof would go through for all $A\in (-1, 1)$.
\end{remark}

\begin{remark}
Note that we could have proved an analogous result with the
function~$G$ defined on $H^{k+1}_{\rm odd}\times H^k_{\rm even}$, with
any $k\geq1$.
\end{remark}

Theorem~\ref{T.two} is a straightforward consequence of
Theorem~\ref{main:thm}. Indeed, the latter ensures that, given small
enough $\epsilon $ and $g$ and assuming that the corresponding Crapper
interface $z_A$ does not have any self-intersections, there is a
solution of the equations with the aforementioned values of~$\epsilon$
and~$g$ that depends smoothly on these parameters. Choosing the Crapper
interface $z_A$ to be an $\eta$-splash curve, by continuity we infer that for small
$\epsilon$ and $g$ the resulting interface $z$ must be an $\eta'$-splash
curve, with $\eta'-\eta$ tending to zero as $|\ep|+|g|\to0$. The
result then follows.

\section{Proof of Theorem~\ref{T.one}}
\label{S.one}

In this section, we only consider capillary-gravity water waves, so we
assume throughout that $\epsilon = 0$. The problem is then governed by
the Bernoulli equation alone, and reads: 

\begin{reform}
\label{capgrav}
Find a $2\pi$-periodic, analytic function $f = \theta + i \tau$ defined on the lower half-plane that satisfies 
\be
\label{eq:brn}
q\frac{d\theta}{d\al} + \sinh H\theta + g e^{-H\theta} \Imag I(\theta) - \ka e^{- H\theta} = 0
\ee
on the boundary and tends to zero at infinity. 
\end{reform}
We prove that, in this case, there exist solutions of
Problem~\ref{capgrav} that exhibit splash singularities for all
sufficiently small values of $g$. Moreover, we show that any such
solution defines a solution of \eqref{euler}--\eqref{irrot}. For that
matter, assume we found an $f$ as required in Problem~\ref{capgrav}. Then we obtain a candidate solution $z$ setting     
$$
\frac{dz}{dw} := e^{if}
$$
(cf.\ section \ref{subsec:par}). Note that the periodicity of $f$ and
the Cauchy integral theorem imply that $z:\Cbb_- \to \Omega$ has the form  
$$
z(w)= w + h(w), 
$$ 
where $h$ is $2\pi$-periodic (meaning that $h(w+2\pi)=h(w)$ for all
complex~$w$). The boundary condition at infinity then implies that $h$
tends to some finite limit as $\Imag{w}\rightarrow -\infty$. This
mapping defines a solution of \eqref{euler}--\eqref{irrot}, if it can
be inverted. In this case, its inverse $w= \phi + i \psi$ defines the
complex potential and the complex conjugate of its derivative $dw/dz$
the velocity. If fact, we have the following result:

\begin{lemma}
\label{lem:splash}
 Let $z:\Cbb_-\rightarrow \Omega$ be a solution of the
 capillary-gravity wave problem (see Problem~\ref{capgrav}) such that the curve 
 $$
 \Scal = \{z(\al)\,|\,\al\in\Rbb\}
 $$
 either does not self-intersect or has a splash singularity. Then $z$ is invertible and its inverse
 $$
 w = \phi + i \psi:\Omega \rightarrow \Cbb_-
 $$
 is such that $v = \nabla \phi$ satisfies \eqref{euler}--\eqref{irrot}.
\end{lemma}
\begin{proof}
When the curve does not have any self-intersections, the result is
proved in \cite{BDT}, so let us assume that $\Scal$ is a splash curve. 

%
We can restrict our attention to $\Omega\cap\{|\Real z|<\pi\}$, which we will denote $\Omega_p$.
Let $P:\Omega_p\to \Cbb$ denote the conformal mapping   
\begin{align*}
P(z) = \sqrt{a - e^{-iz}} ,  
\end{align*}
where the branch cut of the complex square root was taken along the negative real axis. The constant $a\geq 0$ is chosen in such a way that the branch cut and $(0,0)$ lie in the vacuum region, with $(0,0)$ lying inside the bubble in the case of splash singularity.

The exponential function maps the interface $\Scal$ to a closed curve
and the vertical lines $(\pm\pi, y)$ to the positive real axis, with
$-\infty$ being mapped to $(0, 0)$. By symmetry, the point where the
splash singularity occurs has to be on the branch cut and neighboring
points will therefore be mapped to different regions by the square
root. More precisely, if $z_0$ is the splash point and $B$ is a small
ball centered at $z_0$, $P(B\cap\Omega_p)$ will consist of two connected
components with disjoint closures.

Let us set 
$$
\t \Scal := P(\Scal).
$$ 
The transformed interface $\t \Scal$ is a closed Jordan curve in the
right half-plane touching the imaginary axis at exactly two
points. Denoting by $\t\Omega$ the bounded connected component of $
\Cbb\setminus \t\Scal$, we see that the restriction  
$$
P: \Omega_p \rightarrow \t\Omega\setminus [\sqrt{a}, \infty)
$$
is a conformal bijection. Note, however, that 
$$
\lim_{z \rightarrow -\infty} \frac{dP}{dz} = 0,
$$
so a priori $P$ cannot be extended to a conformal mapping on
$\t\Omega$. 

To avoid this difficulty, we map the half-strip
$$
S := \Cbb_-\cap  \{|\Real w| <\pi\}.
$$ 
to the unit disc $\Dbb$ cut along the negative real axis via the
function $w : \Dbb\setminus (-1, 0]  \to  S$ defined by
\[
w(\zeta) := i\log \zeta.
\]
Note that derivative of $w$ with respect to $\zeta$ becomes unbounded as $\zeta\rightarrow 0$. However, the composed mapping 
$$
P\circ z \circ w: \Dbb\setminus (-1, 0] \rightarrow \t\Omega\setminus [\sqrt{a}, \infty) 
$$ 
can be extended to a mapping analytic on all of $\Dbb$. In fact, by
periodicity it is continuous over the cut and maps $0$ to
$\sqrt{a}$. Moreover, an elementary computation using the precise
expression for~$z$ shows that the derivative of the composed mapping
tends to a nonzero constant. In particular, we have a mapping, which is analytic on
$\Dbb$, continuous on $\overline{\Dbb}$ and injective on $\pa
\Dbb$. This implies it actually has to be a conformal bijection $\Dbb
\rightarrow \t \Omega$ by \cite[Corollary 2.10]{pomm}.
\end{proof}

We are now ready to prove the main result of this section. Arguing as in the
previous section, this result readily yields Theorem~\ref{T.one}.

\begin{proposition}
There exists a $g_0>0$, such that for every $|g|\leq g_0$ there exists
an $H^3$ solution of the capillary-gravity wave problem (Problem~\ref{capgrav}) such that the corresponding interface exhibits a splash singularity.
\end{proposition}

\begin{proof}
Since we are only concerned with solutions to
$$
G_1(g, \ka, A ;\, \theta) = q(A)\frac{d\theta}{d\al} + \sinh H\theta + g e^{-H\theta} \Imag I(\theta) - \ka e^{- H\theta} =0,
$$ 
Remark \ref{rem:1} ensures that there are solutions to this equation
in a neighborhood of  the point $(0,0,A_0;\,\theta_{0})$, where we set $\theta_0 :=
\theta_{A_0}$. Omitting the dependence on $\ka$ for notational
simplicity, for small enough $g_0$ and $\de$ we obtain a unique smooth mapping $\Theta$ which maps
$|g|\leq g_0$, $|A - A_0|\leq \de$ to a neighborhood of $\theta_0\in
H^3_{\rm odd}$ and solves the equation 
$$
G_1(g,A;\, \Theta(g,A)) = 0.
$$
By uniqueness, this mapping essentially coincides with mapping found in Theorem
\ref{main:thm} on their common domain of definition.

It is well-known that the interface of the Crapper solution with parameter $A_0$ has a
splash point, while for values of $A$ slightly smaller (resp.\ larger) than
$A_0$ the curve does not self-intersect (resp.\ intersects
transversally at exactly two points). By continuity and transversality, and shrinking the neighborhoods if necessary, we can
ensure that the corresponding curve $z(\al)$ remains injective for
$\Theta(g, A_0 - \delta)$ and any $|g|\leq g_0$, while it
self-intersects transversally and at exactly two points for $\Theta(g, A_0 + \delta)$
and all $|g|\leq g_0$. 

In fact, by taking the real part of \eqref{ref:3} we can read    
$$
\theta_0(\al) = 2\arctan\Big(\frac{2A_0 \sin \al}{1-A_0^2}\Big).
$$
Its second derivative $\theta''_0(\alpha)$ is an odd, periodic function, and therefore zero for $\alpha = \{ 0, \pm\pi\}$. On the other hand, a straightforward calculation shows 
$$
\theta_0''(\al)>0, \quad  \al\in(-\pi, 0) \quad \mathrm{and}\quad \theta_0''(\al)<0,  \quad \al\in(0, \pi).
$$ 
The perturbed solution $\Theta(g,A)$ can be chosen arbitrarily close to $\theta_0$ in the $\Ccal^2$-norm, so by shrinking the neighborhoods we can ensure $\Theta''(g,A)$ is strictly positive on $(-\pi + \epsilon, -\epsilon)$ and strictly negative on $(\epsilon, \pi - \epsilon)$ for some arbitrarily small $\epsilon>0$. We choose $\epsilon$ in such a way that the point of self-intersection of the curve $z_{A}$ be contained in $(-\epsilon, \epsilon)$. In particular, in this smaller interval, the curves corresponding to the perturbed solutions intersect the symmetry axis $\{x = 0\}$ on at most two points. By continuity, for each $|g|\leq g_0$, there exists some $A$ with $|A-A_0|<\delta$ such that the curve corresponding 
to $\Theta(g, A)$ exhibits a splash
singularity. Although the connection
between the Crapper solutions with $A>A_0$ and the Euler equations is
uncertain,  Lemma~\ref{lem:splash} ensures that it is indeed a
solution of the Problem~\ref{capgrav}, so the theorem follows.
\end{proof}

\section*{Acknowledgements}

A.E.\ is supported by the ERC grant 633152. The authors are supported in part by the ICMAT Severo
Ochoa grant SEV-2011-0087 and by the grants of the Spanish MINECO MTM2011-26696 (D.C.\ and
N.G.) and FIS2011-22566 (A.E.).

\end{document}